\documentclass[11pt,fleqn]{article}
\input psfig.sty

\makeindex

\newcommand{\EMPTYSET}{\mbox{$\O$}}

\newcommand{\HHS}{\mbox{$\hspace{2pt}$}}
\newcommand{\CL}{\mbox{$\varphi$}}
\newcommand{\OPER}{\mbox{$\alpha$}}

\newcommand{\NOT}{\mbox{$-$}}
\newcommand{\INV}{\mbox{$^{-1}$}}

\newcommand{\CALA}{\mbox{${\cal A}$}}

\newcommand{\CALC}{\mbox{${\cal C}$}}

\newcommand{\CALN}{\mbox{${\cal N}$}}

\newcommand{\CALS}{\mbox{${\cal S}$}}

\newcommand{\COMP}{\mbox{${\HHS \bf \cdot} \HHS$}}

\newcommand{\GEN}{\mbox{$\gamma$}}

\newcommand{\CAT}{\mbox{${\cal C}$}}

\newcommand{\ARROW}{\mbox{$\rightarrow$}}

\newcommand{\SET}{\mbox{${\bf Set}$}}
\newcommand{\POW}{\mbox{${\bf Pow}$}}

\newcommand{\POSET}{\mbox{${\bf Porder}$}}

\newcommand{\DOM}{\mbox{${\bf Dom}$}}

\newcommand{\OPR}{\mbox{${\bf Opr}$}}

\newcommand{\MCONT}{\mbox{${\bf MCont}$}}

\newcommand{\TRANS}{\mbox{${\bf Trans}$}}

\newcommand{\UTRANS}{\mbox{${\ \stackrel{u}{\longrightarrow}\ }$}}
\newcommand{\ATRANS}{\mbox{${\ \stackrel{\alpha}{\longrightarrow}\ }$}}
\newcommand{\FTRANS}{\mbox{${\ \stackrel{f}{\longrightarrow}\ }$}}
\newcommand{\FPTRANS}{\mbox{${\ \stackrel{f'}{\longrightarrow}\ }$}}

\newcommand{\GTRANS}{\mbox{${\ \stackrel{g}{\longrightarrow}\ }$}}
\newcommand{\GPTRANS}{\mbox{${\ \stackrel{g'}{\longrightarrow}\ }$}}
\newcommand{\FGTRANS}{\mbox{${\ \stackrel{f.g}{\longrightarrow}\ }$}}

\newcommand{\HXTRANS}{\mbox{${\ \stackrel{h_X}{\longrightarrow}\ }$}}

\newcommand{\HYTRANS}{\mbox{${\ \stackrel{h_Y}{\longrightarrow}\ }$}}

\newcommand{\HTRANS}{\mbox{${\ \stackrel{h}{\longrightarrow}\ }$}}
\newcommand{\KTRANS}{\mbox{${\ \stackrel{k}{\longrightarrow}\ }$}}
\newcommand{\FHTRANS}{\mbox{${\ \stackrel{f.h}{\longrightarrow}\ }$}}
\newcommand{\KGTRANS}{\mbox{${\ \stackrel{k.g}{\longrightarrow}\ }$}}

\newcommand{\NBHD}{\mbox{$\eta$}}
\newcommand{\REG}{\mbox{$\Delta$}}
\newcommand{\DEL}{\mbox{$\Delta$}}

\newtheorem{theorem}{Theorem}[section]

\newtheorem{corollary}[theorem]{Corollary}

\newtheorem{proposition}[theorem]{Proposition}

\newenvironment{proof}{\small {\bf Proof:} }{}

\newcommand{\qed}{\ }

\begin{document}

\bibliographystyle{plain}
\bibstyle{plain}

\title{
	{\bf Domination and Closure}
	}
   
\author{
   John L. Pfaltz\\
   Dept. of Computer Science,  University of Virginia \\
    {\tt jlp@virginia.edu } \\
   }

\maketitle

\begin{abstract}
An expansive, monotone operator is dominating;
if it is also idempotent it is a closure operator.
Although they have distinct properties, these two kinds of 
discrete operators are also intertwined.
Every closure is dominating; every dominating operator 
embodies a closure.
Both can be the basis of continuous set transformations.
Dominating operators that exhibit categorical pull-back
constitute a Galois connection and must be antimatroid 
closure operators.
Applications involving social networks and learning spaces
are suggested.
\end{abstract}

\noindent
{\bf keywords:} antimatroid, operator, pull-back, category, closure, domination.
\section {Introduction} \label{I}

%

The concept of ``domination'' is an important one in graph theory, where a
set of nodes ``dominates'' its neighbors.
An extensive treatment can be found in
\cite{HayHedSla98b,HayHedSla98,Sum90},
and an interesting historical application in
\cite{ReVRos00}.
Concepts of ``closure'' arise in many contexts, including topology,
algebra,
and its closely related concept of ``convexity''
\cite{Chv09}.

This paper develops both concepts in terms of general, discrete set systems,
where $\REG$ and $\CL$ are dominating and closure operators
respectively.
In Section \ref{CO}, we develop the connection between domination
and closure that seems to be largely unexplored.
Every domination operator gives rise to a closure operator, and every
closure operator is dominating.

Section \ref{T} explores the properties of domination and closure under
functional transformation, especially continuous transformation, and
briefly reviews the definition of closure in terms of Galois transformations,
or connections.

It is natural to regard collections of functions, whether operators or
transformations, as a category. 
In Section \ref{CC}, we develop this theme and introduce the 
category $\DOM$ to denote all domination
morphisms over discrete sets of $S$.
We show that if a subcategory $\CALC \subset \DOM$ exhibits the 
pull-back property, then it consists of only antimatroid closure operators.

We believe this approach to the study of set systems, including
directed and undirected networks, solely in terms of set-valued operators
and the representation of set system dynamics by 
set-valued transformations from one discrete  set system, $\CALS$,
to another, $\CALS'$,
may be original.



\section {Set Systems} \label{SS}

Let $S$ be any finite set.
By a set system, $\CALS$, we mean a collection of subsets of $2^S$, the power
set of $S$, together with various operators, $\alpha, \beta, \ldots$
defined on this collection.
Elements of the ground set, $S$, we denote
by lower case letters $x, y, z$.
In the general theory, the nature of these members is unimportant; 
although in some
applications
they can be significant.

The sets of $\CALS$ are denoted by upper case letters, $X, Y, Z$.
We assume that $S = \bigcup_{X \in \CALS} X$.
Sets can also be denoted by their constituent members, such as $\{x, y, z\}$,
or more simply by $xyz$.
One can regard $xyz$ to be the label of a set.
Whenever we reference an element, such as $x$, we will actually be denoting
the singleton set $\{x\}$.
And if we use a familiar expression, such as $x \in Y$, it can be interpreted
as $\{x\} \ \cap \ Y \neq \EMPTYSET$.
This is a paper about sets, their properties and their transformations.
The cardinality of a set $Y$ is denoted $| Y |$.

\subsection{Operators on $\CALS$} \label{OS}

An unary {\bf operator} $\OPER$ on $\CALS$ is a function defined on the sets of $\CALS$, 
that is, for all $Y \in \CALS$, $Y.\OPER \in \CALS$.
Operators are expressed in suffix notation because they are ``set-valued''.\footnote
	{
	Here we follow a convention that is more often used by algebraists
	\cite{ZarSam58}.
	Set valued functions/transformations, $f$, are presented in suffix notation,
	$e.g.$ $S' = S.f$; single valued functions, $f$, on set elements are denoted
	by prefix notation, $e.g.$ $e' = f(e)$.
	}
An operator $\OPER$ is said to be:
\\
\hspace*{0.4in}
{\bf contractive}, if $Y.\OPER \subseteq Y$;
\\
\hspace*{0.4in}
{\bf expansive}, if $Y \subseteq Y.\OPER$;
\\
\hspace*{0.4in}
{\bf monotone}, if $X \subseteq Y$ implies $X.\OPER \subseteq Y.\OPER$;
\\
\hspace*{0.4in}
{\bf idempotent}, if $Y.\OPER.\OPER = Y.\OPER$;
\\
\hspace*{0.4in}
{\bf path independent} if $(X.\OPER \cup Y.\OPER).\OPER = (X \cup Y).\OPER$.

\noindent
Contractive operators are often called {\bf choice operators}
\cite{Kos99,Sen77}.
Path independent choice operators are important in economic theory
\cite{Johdea01,MonRad01}.
Operators, $\REG$, that are expansive and monotone we call
{\bf domination} (or dominating) operators.
They are central to this paper.

If $X.\OPER = Y.\OPER = Z$ then $X$ and $Y$ are said to be {\bf generators} of $Z$ (with
respect to $\OPER$).
A  set $Y$ is said to be a {\bf minimal} generator of $Z$ if for all $X \subset Y$,
$X.\OPER \neq Z$.
The operator $\OPER$ is said to be {\bf uniquely generated} if for all $Z$,
$Y.\OPER = Z$ implies there exists a unique minimal generator $X \subseteq Y$, such that $X.\OPER = Z$.
When $X.\OPER = Y.\OPER$, we say $X$ and $Y$ are {\bf $\OPER$-equivalent},
denoted $X =_{\OPER} Y$.
 
\subsection {Extended Operators} \label{EO}

An operator $\OPER$ is said to be {\bf extended} if for all $Y \in \CALS$,
$Y.\OPER = \bigcup_{y \in Y} \{y\}.\OPER$.
That is, $\OPER$ has been extended from its definition on singleton subsets.
This is, perhaps, the most common way of defining set-valued operators.
It is not difficult to show that:

\begin{samepage}
\begin{proposition}\label{p.EX1}
If $\OPER$ is an extended operator, then for all $z \in Y.\OPER$, there exists
$y \in Y$ such that $z \in \{y\}.\OPER$.
\end{proposition}
\end{samepage}

\begin{samepage}
\begin{corollary}\label{c.EX2}
If $\OPER$ is an extended operator, then $\OPER$ is a monotone operator.
\end{corollary}
\end{samepage}

\begin{samepage}
\begin{corollary}\label{c.EX3}
If $\OPER$ is an extended operator, then $\EMPTYSET.\OPER = \EMPTYSET$.
\end{corollary}
\end{samepage}
\medskip

If an operator $\OPER$ is extended, then it can be visualized as a simple graph with
$(x, y) \in E$ if and only if $y \in \{x\}.\OPER$.
We call this a {\bf graphic representation}.

\subsection{Dominating Operators} \label{DO}

Throughout this paper we concentrate on $expansive$, $monotone$ operators which
we called  domination 
operators, $\REG$, in Section \ref{OS}.
We call $Y.\REG$ the {\bf region} dominated by $Y$.
If $S$ denotes the nodes of a network $\CALN$, one can define
$X.\REG = X \cup \{ y | \exists x \in X, (x, y) \in E \}$.
$X$ is said to ``dominate'' $X.\REG$ and there is a large literature, called
``domination theory'', devoted to the combinatorial properties of the minimal
generators, $X$, when
$X.\REG = S$
\cite{HayHedSla98}.
Whence the term ``domination'' operator.
However, domination operators need not be graphically representable.  
Consider the following $\REG$ defined on $S = \{a, b, c, d\}$.
For the singleton sets, let
$\{a\}.\REG = \{ac\}$,
$\{b\}.\REG = \{bc\}$,
$\{c\}.\REG = \{cd\}$, and
$\{d\}.\REG = \{d\}$.
Except for $\{ab\}$, let $Y.\REG = \bigcup_{y \in Y}\{y\}.\REG$, but
let $\{ab\}.\REG = \{abcd\}$.
This is not a simple extension of $\{a\}.\REG$ and $\{b\}.\REG$;
yet $\REG$ is well defined.
It should not be surprising that a larger set might have a 
larger radius of domination.

If the expansive, monotone operator, $\REG$, is also idempotent, 
it is called a
{\bf closure operator}, $\CL$.

It is sometimes convenient to distinguish that part of a dominated region $Y.\REG$
from its generator $Y$.
We call $Y.\NBHD = Y.\REG \NOT Y$ the dominated {\bf neighborhood} of $Y$.
Observe that, as an operator, $\NBHD$ is not expansive and generally is not 
monotone.


%
%
%

\section {Closure Operators} \label{CO}

The concept of closure appears to be an important theme in many discrete systems
\cite{Pfa08,Pfa13}.
A closure system can
be defined by simply enumerating a collection, $\CALC$, of sets which are said to
be {\bf closed}.
The union of all the subsets of $\CALS$ is assumed to be in $\CALC$.
The only other constraint is that if the sets $X$ and $Y$ are in $\CALC$, $i.e.$ are
closed, then $X \ \cap \ Y \in \CALC$, must be closed.

Given such a collection, $\CALC$, of closed sets
we can then define a {\bf closure operator}, $\CL$, on $\CALS$ by
letting $Y.\CL$ denote the smallest set $C_i \in \CALC$ such that $Y \subseteq C_i$.
Since $\CALC$ is closed
under intersection, $\CL$ is single valued and well
defined.
It is well known
\cite{Ore62,Pfa96}
that this definition of closure is equivalent to
the one given in section \ref{OS}, that is
``an operator $\CL$ is a closure operator if and only if
\\
\hspace*{0.4in}
$Y \subseteq Y.\CL$, expansive;
\\
\hspace*{0.4in}
$X \subseteq Y$ implies $X.\CL \subseteq Y.\CL$, monotone; and
\\
\hspace*{0.4in}
$Y.\CL.\CL = Y.\CL$, idempotent''.
\\
Consequently, every closure operator, $\CL$, is a dominating operator because it's
monotone and expansive.
A dominating operator, $\DEL$, is a closure operator, $\CL$
only if it is idempotent.

For any monotone operator, $\OPER, \DEL$ or $\CL$, we have
\\
\hspace*{0.4in}
	$(X \cap Y).\OPER \subseteq X.\OPER \cap Y.\OPER$,
\\
\hspace*{0.4in}
	$X.\OPER \cup Y.\OPER \subseteq (X \cup Y).\OPER$.
 
\begin{samepage}
\begin{proposition}\label{p.TC}
Let $\CL$ be an idempotent dominating operator $\DEL$.
If $y \in X.\CL$ then $\{y\}.\NBHD \subseteq X.\DEL$.
\end{proposition}
\end{samepage}
\noindent
\begin{proof}
Suppose $\exists X$ and $y, z$, with $y \in X.\CL, z \in \{y\}.\NBHD$, but
$z \not\in X.\DEL$.
Then $z \in X.\DEL.\DEL$ contradicting the idempotency of $\DEL$.
\qed
\end{proof}

\medskip
\noindent
Proposition \ref{p.TC} effectively asserts that for a dominating 
operator, $\DEL$, to be a closure operator, $\DEL$ must be ``transitively
closed''.
Still a third characterization of closure systems can be found in
\cite{MonRad01};
\begin{proposition}\label{prop.CC3}
An expansive operator $\CL$ is a closure operator if and only if $\CL$
is path independent.
\end{proposition}

Let $Y$ be closed, a closure operator, $\CL$, is said to be:
\\
\hspace*{0.4in}
{\bf matroid} \hspace*{0.8in} if $x, z \not\in Y$ then $z \in (Y \cup \{ x \}).\CL$ implies $x \in (Y \cup \{ z \}).\CL$;
\\
\hspace*{0.4in}
{\bf antimatroid} \hspace*{0.5in} if $x, z \not\in Y$ then $z \in (Y \cup \{ x \}).\CL$ implies $x \not\in (Y \cup \{ z \}).\CL$;
\\
\hspace*{0.4in}
{\bf topological} \hspace*{0.58in} if $\EMPTYSET.\CL = \EMPTYSET$ and $(X \cup Y).\CL = X.\CL \cup Y.\CL$.
\\
The first two expressions on the right are also known as the ``exchange'' and ``anti-exchange'' axioms.

Matroids are generalizations of linear independent structures
\cite{KorLovSch91}.
Matroid closure is usually denoted by the ``spanning operator'', $\sigma$,
\cite{Tut71,Wel76}.
Antimatroids are typically viewed as convex geometries,
\cite{Cop98,Jam82},
where closure is the convex {\bf hull} operator, sometimes denoted
by $h$,
\cite{EdeJam85,Jam82}.

A closure operator,$\CL$, is said to be {\bf finitely generated} if every closed
set has finite generators.
Since we assume $S$ is finite, all operators will be finitely generated.
In \cite{PfaSla13},
it is shown that:
\begin{samepage}
\begin{proposition}\label{p.FGEN1}
Let $\CALS$ be finitely generated and let $\CL$ be antimatroid. 
If $X$ and $Y$ are generators of a closed set $Z$,
then $X \cap Y$ is a generator of $Z$.
\end{proposition}
\end{samepage}


\begin{samepage}
\begin{proposition}\label{p.FGEN2}
If $\CALS$ is finitely generated, then $\CL$ is antimatroid if and only if $\CL$ is
uniquely generated.
\end{proposition}
\end{samepage}

\noindent
A counter example is presented in 
\cite{PfaSla13}
to show that the condition of finite generation is necessary.

\subsection{Dominated Closure}  \label{RC1M}

It is evident that
the domination operator $\DEL$ and closure operator $\CL$ are closely related.
For example, if $Y.\DEL.\DEL = Y.\DEL$ for all $Y$, then $\DEL$ is a
closure operator.
But, in general, $Y.\DEL \subset Y.\DEL.\DEL$.
In this section, we explore the close relationship between these two operators
even further.

We define {\bf dominated closure} (sometimes denoted by $\CL_{\Delta}$) to be:
\begin{equation} \label{RCDEF}
Y.\CL_{\Delta} = \bigcup_{Z \subseteq Y.\Delta} \{ Z | Z.\DEL \subseteq Y.\DEL \}
\end{equation}
Here it is apparent that, $Y \subseteq Y.\CL \subseteq Y.\DEL$, as was shown
in Proposition \ref{p.TC}.

Although in general $Y.\DEL.\DEL \neq Y.\DEL$, we have
\begin{samepage}
\begin{proposition}\label{p.CLO.REG}
For all $Y$, $Y.\CL.\DEL = Y.\DEL$.
\end{proposition}
\end{samepage}
\noindent
\begin{proof}
Let $z \in Y.\CL.\DEL$.
If $z \in Y.\CL$, then since $Y.\CL \subseteq Y.\DEL$, we are done.
So assume $\exists y \in Y.\CL$, $z \in \{y\}.\DEL$.
But, $y \in Y.\CL$ implies $\{y\}.\DEL \subseteq Y.\DEL$ so $z \in Y.\DEL$.
By monotonicity, $Y.\DEL \subseteq Y.\CL.\DEL$, so equality follows.
\qed
\end{proof}
\medskip

\noindent
The operators, $\DEL$ and $\CL$ are not in general commutative, since $Y.\CL.\DEL = Y.\DEL \subset Y.\DEL.\CL$
as shown by the following example.
Let $S = \{ abc \}$ and let $Y = \{a\}$, where $\{a\}.\DEL = \{ab\}$, $\{b\}.\DEL = \{bc\}$, so 
$\{a\}.\CL.\DEL = \{a\}.\DEL = \{ab\}$ as postulated by Proposition \ref{p.CLO.REG}, but
if $\{ab\}.\CL = \{abc\}$ then
$\{a\}.\DEL.\CL = \{ab\}.\CL = \{abc\}$.

\begin{samepage}
\begin{proposition}\label{p.RC}
An operator $\CL$ is a closure operator if and only if there exists a dominating
operator, $\DEL$, related to $\CL$ by (\ref{RCDEF}).
\end{proposition}
\end{samepage}
\noindent
\begin{proof}
If $\CL$ is a closure operator, then let $\DEL = \CL$.
Readily, $\DEL$ is monotone, expansive because $\CL$ is.
Let $Z \subseteq Y.\DEL$.
Since $Y.\CL.\CL = Y.\CL$, $Z.\DEL = Z.\CL \subseteq Y.\CL = Y.\DEL$ satisfying
equation (\ref{RCDEF}).
\\
\\
Conversely, let $\DEL$ be any monotone, expansive operator, and let $\CL$ be
defined by (\ref{RCDEF}).
Monotonicity and expansivity follow from $Y \subseteq Y.\CL \subseteq Y.\DEL$.
We need only show idempotency.
Readily, $Y.\CL \subseteq Y.\CL.\CL$.
By Prop. \ref{p.CLO.REG}, $Y.\CL.\DEL \subseteq Y.\DEL$.
So, $Y.\CL.\CL = \bigcup_{Z \subseteq Y.\CL.\Delta} \{ Z | Z.\DEL \subseteq Y.\CL.\DEL \}
\subseteq \bigcup_{Z \subseteq Y.\Delta} \{ Z | Z.\DEL \subseteq Y.\DEL \} = Y.\CL$.
\qed
\end{proof}
\medskip

\begin{samepage}
\begin{proposition}\label{p.REG.GEN}
$X$ is a $\REG$-generator of $Y.\REG$ if and only if $X$ is a $\CL_{\Delta}$-generator of $Y.\CL_{\Delta}$.
\end{proposition}
\end{samepage}
\noindent
\begin{proof}
Suppose $X$ is a $\REG$-generator of $Y.\DEL$, so $X.\REG = Y.\REG$.
By (\ref{RCDEF}),
$X.\CL_{\Delta} = \bigcup_{Z \subseteq X.\REG} \{ Z.\REG \subseteq X.\REG \}$
= $\bigcup_{Z.\subseteq Y.\Delta} \{Z.\REG \subseteq Y.\REG \}$ = $Y.\CL_{\Delta}$.
\\
Conversely, let $X$ be a $\CL_{\Delta}$-generator of $Y.\CL_{\Delta}$ and assume
$X$ is not a $\REG$-generator of $Y.\REG$, so $X.\REG \neq Y.\REG$.
Let $Z_0 = X.\REG \NOT Y.\REG$ (or else $Y.\REG \NOT X.\REG$).
$Z_0 \not\subseteq Y.\REG$ implies
$Z_0 \not\subseteq \bigcup_{Z \subseteq Y.\Delta}\{Z.\REG \subseteq Y.\REG \} = Y.\CL_{\Delta}$
contradicting assumption that $X$ is a $\CL_{\Delta}$-generator of $Y.\CL_{\Delta}$.
\qed
\end{proof}

\begin{samepage}
\begin{proposition}\label{p.C.REG}
A dominating operator, $\REG$, is itself a closure operator, $\CL$, if
and only if $X \subseteq Y.\REG$ implies $X.\REG \subseteq Y.\REG$.
\end{proposition}
\end{samepage}
\noindent
\begin{proof}
Assume the condition holds.
Since $\REG$ is monotone, expansive we need only show $\REG$ is idempotent.
But readily, $Y.\REG.\REG = \bigcup_{X \subseteq Y.\Delta} \{ X | X.\REG \subseteq Y.\REG \}
= \bigcup \{ X \subseteq Y.\REG \} = Y.\REG$.
\\
Conversely, if $\REG$ is idempotent, it is a closure operator by definition.
\qed
\end{proof}
\medskip

Several set systems have the property that $X \subseteq Y.\REG$
implies $X.\REG \subseteq Y.\REG$.
Let $P$ be a partially ordered set and let $Y.\REG = \{ x | x \leq y \in Y \}$.
Readily $X \subseteq Y.\REG$ implies $X.\REG \subseteq Y.\REG$, so
$\REG$ is a closure operator, $\CL$.
It has been called ``downset'' closure
\cite{PfaSla13,Pfa08}.
The maximal elements in $Y$ constitute a unique generator; it is antimatroid.

%

One can use $\REG$ to construct a large variety of closure systems $\CALS$,
As an example, let $P$ be any set, which we will augment with a special element, $*$.
Let $Y \subseteq P$, and define $Y.\REG = Y \cup \{*\}$ and let $\{*\}.\REG = \{*\}$.
Then $Y.\CL = Y \cup \{*\}$ and $\{*\}.\CL = \{*\}$.
One can optionally let $\EMPTYSET.\CL = \{*\}$ or $\EMPTYSET.\CL = \EMPTYSET$.
It is apparent that $X \subseteq Y.\REG$ implies $X.\REG \subseteq Y.\REG$, so
$\REG$ is a closure operator.
No subset of $P$ is closed, and $Y \subseteq P$ is the unique minimal generator of the 
closed set of the form $Y \cup \{*\}$.
Either $\EMPTYSET$ or $\{*\}$ could be the minimal generator of $\{*\}$.
In either case, $\CALS$ is an antimatroid closure space.
We call $\CALS$, so defined, a ``star space''.

For a simple example where $X \subseteq Y.\REG$ need not imply $X.\REG \subseteq Y.\REG$,
consider $S$ = {\bf Z}, the integers. 
Define $\{y\}.\REG = \{ x \leq y |$ if $y$ is even \} and $\{ z \geq y |$ if $y$ is odd \}.
Let $Y = \{1, 2\}$.
Readily, $2 \in \{1\}.\REG = \{i | i \geq 1\}$, but $\{2\}.\REG = \{j | j \leq 2\} \not\subseteq \{1\}.\REG$.
\section {Transformations} \label{T}

A {\bf transformation} $f$ is a function that maps the sets of one
set system $\CALS$ into another set system $\CALS'$.
If $\CALS' = \CALS$, then a transformation is just an operator on $\CALS$.
More often, however, $\CALS'$ has a different internal structure than $\CALS$.
Frequently, $f$ represents ``change'' in a dynamic set system.
Because the domain, and codomain, of a transformation $f$ is a collection of sets,
including the empty set $\EMPTYSET$, an expression such as $Y.f = \EMPTYSET'$
is well defined.
Similarly one can have $\EMPTYSET.f = Y'$.\footnote
	{
	Normally, we do not distinguish between $\EMPTYSET$ and
	$\EMPTYSET'$.
	The empty set is the empty set. 
	We do so here only for emphasis.
	}
Thus we have a functional notation for sets entering, or leaving, a set system altogether.

\subsection {Monotone Transformations} \label{MT}

A transformation $\CALS \FTRANS \CALS'$ is said to be {\bf monotone} if
$X \subseteq Y$ in $\CALS$ implies $X.f \subseteq Y.f$ in $\CALS'$.
Monotonicity seems to be absolutely basic to transformations and is
assumed throughout this paper.
No other property is.
Monotonicity ensures that if $Y.f = \EMPTYSET$ then
for all $X \subseteq Y$, $X.f = \EMPTYSET'$,
and if $\EMPTYSET.f = Y'$ then
for all $Z' \subseteq Y'$, $\EMPTYSET.f = Z'$, so $Z'.f^{-1} = \EMPTYSET$.
Readily,
\begin{samepage}
\begin{proposition}\label{p.MONOTONE}
The composition $f \COMP g$ of monotone transformations is monotone. 
\end{proposition}
\end{samepage}

%



\subsection {Continuous Transformations} \label{CT}

A transformation $\CALS \FTRANS \CALS'$ is said to be {\bf continuous}
with respect to an operator, $\OPER$, or more simply $\OPER$-continuous,
if for all sets $Y \in \CALS$, $Y.\OPER.f \subseteq Y.f.\OPER'$,
\cite{Cas03,Ore46,PfaSla13,Sla04,Sla10}.
In the referenced literature, continuity is only considered with respect
to a closure operator, $\CL$.
This is reasonable; but as the following propositions show, it can be generalized.

\begin{samepage}
\begin{proposition}\label{p.COP}
Let $\OPER$ be any monotone operator.
If $\CALS \FTRANS \CALS'$ and $\CALS' \GTRANS \CALS''$ are monotone 
$\OPER$-continuous transformations then
$\CALS \FGTRANS \CALS''$ is $\OPER$-continuous
\end{proposition}
\end{samepage}
\noindent
\begin{proof}
Since $f$ is continuous w.r.t. $\OPER$, $Y.\OPER.f \subseteq Y.f.\OPER'$.
Since $g$ is monotone, $Y.\OPER.f.g \subseteq Y.f.\OPER'.g$.
And finally $g$ $\OPER$-continuous yields
$Y.\OPER.f.g \subseteq Y.f.\OPER'.g \subseteq Y.f.g.\OPER''$.
\qed
\end{proof}
\medskip

\noindent
That the composition of $\OPER$-continuous transformations 
is continuous when $\OPER$ is a closure operator has already been shown in
\cite{PfaSla13}, where 
a counter example is provided to demonstrate the necessity of having $g$ be monotone.
They also show that
the collection of all monotone, $\CL$-continuous transformations
forms a concrete category, $\MCONT$,
\cite{PfaSla13}.
By Proposition \ref{p.CLO.REG}
$Y.\CL.\REG = Y.\REG \subseteq Y.\REG.\CL$, so $\REG$ is monotone, $\CL$-continuous,
and thus a member of $\MCONT$.
And since $\CL$ is idempotent, $\CL$ is trivially $\CL$-continuous, as well.

\begin{proposition}\label{p.IC4}
Let $\REG$ be a dominating operator and let
$\CALS \FTRANS \CALS'$ be monotone.
Then $f$ is $\REG$-continuous if and only if
$X.\DEL = Y.\DEL$ implies  $X.f.\DEL' = Y.f.\DEL'$.
\end{proposition}
\begin{proof}
Let $f$ be $\REG$-continuous, and let
$X.\REG = Y.\REG$, so
$X =_{\REG} Y$.
By monotonicity and continuity, $X.f \subseteq X.\REG.f = Y.\REG.f \subseteq Y.f.\REG'$.
Similarly, $Y.f \subseteq X.f.\REG'$.
Since $Y.f.\REG'$ is the smallest $\REG$-set containing $X.f$ and
$X.f.\REG'$ is the smallest $\REG$-set containing $Y.f$,
$X.f.\REG' = Y.f.\REG'$.
\\
Conversely, assume $f$ is not $\REG$-continuous.
So there exists $Y$ with $Y.\REG.f \not\subseteq Y.f.\REG'$
Let $X \in Y.\REG\INV$.
$X.f \subseteq X.\REG.f = Y.\REG.f \not\subseteq Y.f.\REG'$,
so $X.f.\REG' \neq Y.f.\REG'$,
contradicting the condition.
\qed
\end{proof}
\medskip

Thus, the image of a generator under a continuous transformation is again a generator.
However, if
$X \in Y.\REG\INV$ is a minimal generator,
Proposition \ref{p.IC4} only shows that $X.f$ is still a generator of
$Y.f.\REG'$;  it need not be minimal.

A transformation $S \FTRANS S'$ is {\bf $\REG$-surjective} if for all $\REG$-sets $Y'$,
there exists a set $Y\subseteq S$ such that $Y.f = Y'$.
\begin{samepage}
\begin{proposition}\label{p.IC2}
Let $f$ be monotone, $\REG$-continuous and $\REG$-surjective, then for all $\REG$-sets
$Y'$ in $\CALS'$, there exists a $\REG$-set $Y$ in $\CALS$ such that
$Y.f = Y'$.
\end{proposition}
\end{samepage}
\noindent
\begin{proof}
Since $f$ is $\REG$-surjective, $\exists Y, Y.f = Y'$.
But, by monotonicity and $\REG$-continuity, $Y.f \subseteq Y.\REG.f \subseteq Y.f.\REG'$ = Y'.
So, $Y.\REG.f = Y'$.
\qed
\end{proof}

\begin{proposition}\label{p.COMP.SUR}
Let $\CALS \FTRANS \CALS'$, $\CALS' \GTRANS \CALS''$
be monotone, $\REG$-continuous transformations.
If both
$f$ and $g$ are $\REG$-surjective, then so is $\CALS \FGTRANS \CALS''$.
\end{proposition}
\begin{proof}
Because the composition of $\REG$-continuous transformations is
$\REG$-continuous, we need only consider surjectivity.
Let $Y''$ be a $\REG$-set in $\CALS''$.  
Since $g$ is surjective, $\exists Y' \in \CALS'$, $Y'.g = Y''$.
Because, $g$ is continuous we may assume, by Prop. \ref{p.IC2}, that
$Y'$ is an $\REG$-set.
Thus, by surjectivity of $f$, $\exists Y \in \CALS, Y.f = Y'$
Consequently,
$f\COMP g$ is $\REG$-surjective.
\qed
\end{proof}
\medskip

%
A transformation $\CALS \FTRANS \CALS'$ is said to be {\bf $\OPER$-preserving}
if $Y.\OPER.f = Y.f.\OPER'$.
An $\OPER$-preserving map takes $\OPER$-sets onto $\OPER$-sets.


\begin{samepage}
\begin{proposition}\label{p.AP}
Let $\REG$ be a dominating operator and let $f$ be monotone.
$f$ is $\DEL$-preserving if and only if
for all $Y$, $Y.f.\REG' \subseteq Y.\REG.f$.
\end{proposition}
\end{samepage}
\noindent
\begin{proof}
Assume, $Y.f.\REG' \subseteq Y.\REG.f$.
Let $Y = Y.\REG$, so $Y.f.\REG' \subseteq Y.\REG.f = Y.f$.
Readily $Y.f \subseteq Y.f.\REG'$ so $Y.f = Y.f.\REG'$ and $f$ is $\REG$-preserving
\\
\\
Now assume $\REG$ is idempotent and that $f$ is $\REG$-preserving.
By monotonicity of $f$, $Y \subseteq Y.\REG$ implies $Y.f \subseteq Y.\REG.f$ and
and since $\REG$ is monotone, $Y.f.\REG' \subseteq Y.\REG.f.\REG'$.
Idempotency implies $Y.\REG.\REG = Y.\REG$.
Since $f$ is $\REG$-preserving, $Y.\REG.f.\REG' = Y.\REG.f$ so
$Y.f.\REG' \subseteq Y.\REG.f$.
\qed
\end{proof}
\medskip

\noindent
Proposition \ref{p.AP} is also proven in
\cite{PfaSla13}
where $\DEL$ is a closure operator $\CL$.

\begin{samepage}
\begin{corollary}\label{c.AP}
A monotone transformation $f$ is both $\REG$-continuous and $\REG$-preserving,
if and only if $Y.\REG.f = Y.f.\REG'$.
\end{corollary}
\end{samepage}

\subsection {Galois Connections}  \label{GC}

Monotone transformations $S \FTRANS S'$ can provide another mechanism for defining
closure operators on $S$.

Let $S \FTRANS S'$ and $S' \GTRANS S$ be monotone transformations.
The composite $(f \COMP g)$ is called a {\bf Galois connection} if for all
$X \subseteq S, Y' \subseteq S'$,
\\
\hspace*{0.5in}
(1) $X \subseteq X.f.g$	 \hspace*{0.55in} $(f \COMP g)$ is expansive
\\
\hspace*{0.5in}
(2) $Y'.g.f \subseteq Y'$ \hspace*{0.5in} $(g \COMP f)$ is contractive.

\begin{samepage}
\begin{proposition}\label{p.GC1}
Let $S \FTRANS S' \GTRANS S$ be monotone.
The following are equivalent statements.
\\
\hspace*{0.5in}
(a)
$(f \COMP g)$ is a Galois connection.
\\
\hspace*{0.5in}
(b)
For all $X \subseteq S$ and all $Y' \subseteq S'$, $X.f \subseteq Y'$ if and only if $X \subseteq Y'.g$.
\end{proposition}
\end{samepage}
\noindent
\begin{proof}
(a) implies (b):\ \  Let $X.f \subseteq Y'$, so $g$ monotone implies $X.f.g \subseteq Y'.g$,
thus $X \subseteq X.f.g \subseteq Y'.g$.
Similarly, $X \subseteq Y'.g$ implies $X.f \subseteq Y.g.f \subseteq Y'$.
\\
(b) implies (a):\ \ 
Let $X.f = Y'$, so trivially $X.f \subseteq Y'$. 
By (b) $X \subseteq Y'.g$ implying $X \subseteq Y'.g \subseteq X.f.g$.
$f \COMP g$ is expansive.
Similarly, $X \subseteq Y'.g$ implies $Y.g.f \subseteq Y'$.
\qed
\end{proof}

\begin{samepage}
\begin{proposition}\label{p.GC2}
If $S \FTRANS S' \GTRANS S$ is a Galois connection, then
$f$ and $g$ uniquely determine each other.
\end{proposition}
\end{samepage}
\noindent
\begin{proof}
Let $f \COMP g$ be a Galois connection and suppose there exists $h$ such
that $f \COMP h$ is also a Galois connection.
We apply $h$ to $Y'.g.f \subseteq Y'$.
$Y'.g \subseteq Y'.g.f.h$ because $f \COMP h$ is expansive, and
$Y'.g.f.h \subseteq Y'.h$ since $h$ is monotone.
So, $Y.g \subseteq Y'.h$.
Applying $g$ to $Y'.h.f \subseteq Y'$ yields $Y'.h \subseteq Y'.g$,
$\forall Y'$, so $h = g$.
\\
A similar argument shows that $f$ must be unique given $g$.
\qed
\end{proof}

\begin{samepage}
\begin{proposition}\label{p.GC3}
If $S \FTRANS S' \GTRANS S$ is a Galois connection,
then $f \COMP g \COMP f = f$ and $g \COMP f \COMP g = g$.
\end{proposition}
\end{samepage}
\noindent
\begin{proof}
Let $Y' = X.f$.
Since $g \COMP f$ is contractive, $X.f = Y' \subseteq Y'.g.f = X.f.g.f$.
However, $X \subseteq X.f.g$ implies $X.f \subseteq X.f.g.f$ by monotonicity,
so $X .f = X.f.g.f, \forall X$.
\\
Similarly, $Y'.g = Y.g.f.g, \forall Y' \subseteq S'$.
\qed
\end{proof}

\begin{samepage}
\begin{corollary}\label{c.GC4}
If $f \COMP g$ is a Galois connection, then $S \FGTRANS S$ is a closure operator
on $S$.
\end{corollary}
\end{samepage}
\noindent
\begin{proof}
Since $f \COMP g$ are Galois connected, $f \COMP g$ is expansive.
$f$ and $g$ are monotone.
Prop. \ref{p.GC3} establishes that $f \COMP g$ is idempotent.
\qed
\end{proof}
\medskip

The preceding three propositions largely follow the development of Galois
connections provided by Castellini in
\cite{Cas03},
except for changed notation.

A somewhat different development can be found in Ganter \& Wille
\cite{GanWil99}.
They choose to let $f$ and $g$ be ``anti-monotone'', that is,
$X_1 \subseteq X_2$ implies $X_1.f \supseteq X_2.f$ and
$Y'_1 \subseteq Y'_2$ implies $Y'_1.g \supseteq Y'_2.g$
and 
$X \subseteq X.f.g$, $Y' \subseteq Y'.g.f$.
With this definition of Galois connection, Proposition \ref{p.GC1}
must be rewritten as:
``$f \COMP g$ is a Galois connection if and only if
$Y' \subseteq X.f$ implies $X \subseteq Y'.g$ and conversely''.
However, both approaches will yield Proposition \ref{p.GC3} and
its corollary.
With the latter definition, both $f \COMP g$ and $g \COMP f$
are expansive operators on $S$, so both are closure operators.
This is important for the subsequent development of ``Formal Concept Analysis''
in \cite{GanWil99}.

We prefer the development presented here, and in 
\cite{Cas03},
because we can compose monotone transformations, so is easy to show that 
\begin{samepage}
\begin{proposition}\label{p.GC5}
Let $S \FTRANS S' \GTRANS S$ and $S' \HTRANS S'' \KTRANS S'$ be Galois
connections,
then $S \FHTRANS S'' \KGTRANS S$ is a Galois connection.
\end{proposition}
\end{samepage}

%
%

Let $Y.\GEN$ denote the collection of minimal generators of $Y.\DEL$.
If $\DEL$ is uniquely generated, then $Y.\GEN$ is a well defined function.
\begin{samepage}
\begin{proposition}\label{p.UGGC}
If $\DEL$  is uniquely generated, then $\DEL$ and $\GEN$ constitute a Galois connection
on $S$.
\end{proposition}
\end{samepage}
\noindent
\begin{proof}
$Y.\GEN.\DEL = Y.\DEL$, so $\GEN.\DEL$ is expansive.
Because $\DEL$ is uniquely generated, $Y.\GEN \subseteq Y$,
so $Y.\DEL.\GEN$ is contractive.
\qed
\end{proof}
\medskip

\noindent
Hence, by Corollary \ref{c.GC4} any uniquely generated dominating operator
$\DEL$ is a closure operator, and by Proposition \ref{p.FGEN2}, it is 
antimatroid.
If we suppose $\DEL$ is not uniquely generated, say $X$ and $Z$ are minimal
sets such that $X.\DEL = Z.\DEL$, then
$X.\DEL.\GEN = X \cup Z \not\subseteq X$.
$\DEL.\GEN$ is not contractive.

\section {Categorical Closure} \label{CC}

Operators, such as $\OPER, \REG, \CL$, can be naturally regarded
as morphisms in a categorical sense.
In this section we develop this way of approaching domination
and closure.
It is somewhat different from the approach found in 
\cite{Cas03},
but is compatible with it.
The following material has been largely derived from 
\cite{HerStr79,Pie91};
the only substantive differences have been to change notation.
All functions, or morphisms, will still be denoted using postfix notation.
Proposition \ref{p.ANTIMATROID} is, we believe, original.

\subsection{Review and Examples}

Two categories which are usually presented in category text books are
$\SET$ and $\POSET$.
For us,
the objects of $\SET$ will be all possible finite sets.\footnote
	{
	To avoid logical paradoxes, MacLane,
	\cite{Mac98},
	lets the objects in $\SET$ be all the sets in a fixed universe $U$.
	These he calls $small$ sets.
	Readily, our finite sets are small.
	}
The morphisms of $\SET$ consist of all functions $f:X \ARROW Y$
(or morphisms $X \FTRANS Y$).
with the usual composition

In the partial order category $\POSET$, the objects consist of all finite sets, $S$,
with a partial order $\leq_S$ on $S$.
Its morphisms are all order preserving (monotone)
functions, $f:S \ARROW T$, such that $x \leq_S y$ implies
$x.f \leq_T y.f$.
Checking that 
$f \COMP g$ is order preserving is usually left as an
exercise.

These two well known categories are paradigms for the dominance and closure 
categories.
Analogous to $\SET$, is the category $\POW$ whose objects consist of
all finite power sets.
Its morphisms are all total
functions $2^S \FTRANS 2^T$ with the usual composition;
$(\forall X \in 2^S) [X.(f\COMP g) = X.f.g]$ and the
usual identity on $2^S$,
$(\forall X \subseteq S) [X.id_S = X]$.
It is not difficult to show this composition is well-defined and associative.
The transformations of Section \ref{T} are morphisms in $\POW$.

The category $\POW$ would appear to be completely isomorphic to $\SET$,
with each $S \in \SET$ replaced by $2^S \in \POW$.
But there are essential differences.
An object $Y$ is said to be {\bf terminal} in $\CAT$
if for every object $X$ in $ \CAT$ there exists exactly one morphism
$X \longrightarrow Y$.
As is well known,
\cite{Mac98},
every element $x$ is a terminal object of $\SET$, since for
all $S$, the function $f:S \ARROW x$ must be unique.
However, the singleton sets in $\POW$ need not be terminal.
To see this, let $Y = \{x, y\} \in 2^S$ and let $f, g$ be 
extended transformations, where $\{x\}.f = \{x\}$,
$\{y\}.f = \EMPTYSET$ and
$\{x\}.g = \EMPTYSET$, $\{y\}.g = \{x\}$.
Then $S.f = S.g = \{x\}$, but $f \neq g$.
So the singleton elements $\{x\}$ cannot be terminal in $\POW$.
Only $\EMPTYSET$ is terminal in $\POW$, and initial as well.

%
 
If we restrict the morphisms in $\POW$ to be monotone, or order preserving,
as in Section \ref{MT},
that is, $X \subseteq Y$ implies $X.f \subseteq Y.f$ then we have
an exact analogue to $\POSET$, but over $\POW$, not $\SET$.
This category, which we call $\TRANS$, has the objects of $\POW$ as
its objects, and the collection of all monotone
transformations $2^S \FTRANS 2^T$.
It is simply a restriction on the morphisms of $\POW$.
Proposition \ref{p.MONOTONE} asserts that composition in $\TRANS$ is
still well defined.
Much of Section \ref{T} is simply concerned with examining the
properties of transformations $f$ in $\TRANS$.
%
%
When the morphisms in $\TRANS$ are of the form $2^S \ATRANS 2^S$, which
we have called ``operators'', they constitute a category $\OPR \subset \TRANS$.
%

Our final example is that of $\DOM$, the category of all 
expansive, monotone operators, $\REG$ on  power sets
$2^S$; that is, the category of dominating operators.\footnote
	{
	Note that ``$dom$'' is the standard categorical way of denoting
	the $domain$ of a morphism, that is, if $\OPER:2^S \ARROW 2^T$
	then $dom\ \OPER = 2^S$ and $cod\ \OPER = 2^T$, where $cod$ 
	denotes $codomain$.
	We do not use the terminology $dom$ or $cod$ in this paper.
	}
The objects of $\DOM$ are precisely those of $\POW$, that is
$2^S$ partially ordered by $\subseteq$.
Its morphisms are all expansive, monotone (order preserving) functions (operators)
$\OPER : (S, \subseteq) \ARROW (S, \subseteq)$,
such that $X \subseteq Y$ implies $X.\OPER \subseteq Y.\OPER$ and $X \subseteq X.\OPER$.
Readily, the identity operator $id_S$ preserves the order $\subseteq$.
The usual composition, $\OPER \COMP \beta$ is order preserving.
Readily, $\DOM \subset \OPR \subset \TRANS \subset \POW$.

\subsection {Pullbacks} \label {PULL}

Recall that
the {\bf pullback} of a pair of morphisms, $Y \FTRANS Z$ and $X \GTRANS Z$,
is an object $W$ and two morphisms $W \FPTRANS X$, $W \GPTRANS Y$ such
that $g' \COMP f = f' \COMP g$.
Moreover, if 
there exist morphisms $V \HXTRANS X$, $V \HYTRANS Y$,
such that $h_Y \COMP f = h_X \COMP g$ then
there exists a $unique$ morphism
$V \UTRANS W$
such that $h_Y = u \COMP g'$ and $h_X = u \COMP f'$.

\begin{samepage}
\begin{proposition}\label{p.ANTIMATROID}
Let $\CAT \subset \DOM$ be a subcategory.
Its morphisms $\DEL$ are 
antimatroid closure operators if and only if $\CAT$
exhibits the pullback property of Figure \ref{PULLBACK2}.
\end{proposition}
\end{samepage}
\noindent
\begin{proof}
Suppose that the dominating morphisms $\REG$ of $\CAT \subseteq \DOM$
are antimatroid closure operators.
We must show that the pullback diagram of Figure \ref{PULLBACK2} is 
satisfied.
We assume that $\REG_X:X \ARROW Z$ and $\REG_Y:Y \ARROW Z$, so $X$ and $Y$ are 
generators of $Z$.
Let $V$ be any set $V \subseteq_X X$, $V \subseteq_Y Y$ such that
$V.(\subseteq_X \COMP \REG_X) = Z$ and $V.(\subseteq_Y \COMP \REG_Y) = Z$, so
$V$ is a generator of $Z$.
By Prop. \ref{p.FGEN1}, $W = X \cap Y$ is the unique pullback of these 
two generators.
\\
Conversely, let the morphisms $\DEL \in \CAT$ exhibit the pullback property
of Figure \ref{PULLBACK2}.
$X$ and $Y$ are generators of $Z$.
Since $V \UTRANS W$ is unique, by Prop. \ref{p.FGEN2}, $\DEL$ is uniquely generated.
By Prop. \ref{p.UGGC}, $\DEL$ must be a closure operator.
\qed
\end{proof}
\begin{figure}[ht]
\centerline{\psfig{figure=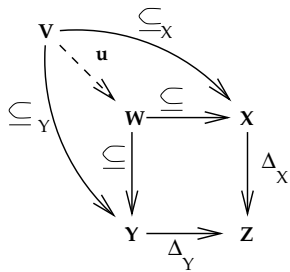,height=1.4in}}
\caption{The pullback of antimatroid generators.
\label{PULLBACK2} }
\end{figure}

\noindent
Observe that in Figure \ref{PULLBACK2}, the sets $V, W, X, Y$ are all generators of $Z$.
It is known that pullback diagrams preserve monomorphisms and retractions
\cite{Mac98}.

It is conjectured that an analog of Proposition \ref{p.ANTIMATROID} is true as well,
that is that a subcategory of $\DOM$ satisfying the ``push-out'' property must consist
of matroid closure operators.

\section {Summary} \label{S}

There are many more dominating than closure operators,\footnote
	{
	If $|S| = n \geq 10$, there exist more than $n^n$ distinct
	antimatroid closure operators
	\cite{Pfa95c}.
	}
in spite of Proposition \ref{p.RC} which established that 
for every dominating operator there exists
a corresponding (not necessarily unique) closure operator.
While dominating operators are more ubiquitous, for example most network
operators associated with internet analysis are expansive;
closure operators are more structured.
By Proposition \ref{p.UGGC}, only antimatroid closure operators can be
uniquely generated.
However, dominated closure has been used to reduce social networks in a
way that preserves path connectivity
\cite{Pfa13}.

The dominated closure,
$\CL_{\REG}$ defined by (\ref{RCDEF}) is often called a {\bf neighborhood closure},
because the definition (\ref{RCDEF}) can be rewritten as 
\begin{equation} \label{RCDEF2}
Y.\CL = Y \cup \bigcup_{\{z\} \subseteq Y.\NBHD} \{ \{z\} | \{z\}.\REG \subseteq Y.\REG \}
\end{equation}
which, in many cases, is computationally much more efficient.

Dominating operators can be easily computed.
A common way, in practice, of defining a dominating operator and its
dominated neighborhood is by an adjacency matrix, $\CALA$.
For each row $i$ of $\CALA$, if $(i,k) \neq 0$, then $k$ is in the region dominated by $i$, or
$\{k\} \subseteq \{i\}.\REG$.
Thus $\CALA$ can be the base of a graphic representation, $G$,  described in Section \ref{EO}.
If $\REG$, and $\CALA$, are symmetric, the graph $G$ is undirected.
In most applications, $Y.\REG$ is assumed to be $\bigcup_{\{y\} \subseteq Y} \{\{y\}.\REG \}$,
or all elements directly connected to $Y$ in $G$, $i.e.$ $\REG$ is an extended operator.

However, one might want much larger sets $Z$ to have a much wider scope of dominance.
Similarly, in $\CALA \times  \CALA$, $(i, k) \neq 0, $ if $k$ is two, or fewer, links from $i$.
It is still a dominance region.
But, it is not graphically representable, nor is it extensible.

The term ``domination'' is well established in the graph theory literature, but we
prefer to think of these expansive, outward accessing operators as ``exploratory''
operators, especially when viewing their role in social network analysis.
Some educators have suggested that ``knowledge spaces'' might be modelled
by closed sets.
If so, ``exploratory'' closure, $\CL_{\Delta}$, becomes a plausible mechanism for ``learning''.
One can regard $S$ as a universe of related ``experiences'' and ``knowledge'' to be a 
collection $Y$ of such experiences, or skills.
An experience $f$ expands one's knowledge if it is congruent with one's existing knowledge
set $Y$;
that is if $\{f\}.\NBHD \subseteq Y.\NBHD$.
That is, the connections of $f$ are congruent with the connections of $Y$; it makes sense.

Dominated closure may find other applications as well.
In
\cite{Pfa13c},
it was shown that one can define ``fuzzy'' dominated closure, yet still retain
many of the crisp properties of closure operators.

It is evident that many of the preceding results, which have been expressed
in operator terminology, can be recast using graph terminology.
In return, graph theory can
provide a rich source of discrete set systems.
In particular, domination
\cite{HayHedSla98b,HayHedSla98,Sum90}
and closure
\cite{Cas03,FarJam86,JamPfa05,Ore46}
have been well studied and can provide many operator examples.

{\small
\bibliography{/home/jlp/MATH.d/BIB.d/mathDB,/home/jlp/MATH.d/BIB.d/pfaltzDB,/home/jlp/MATH.d/BIB.d/otherDB,/home/jlp/MATH.d/BIB.d/orlandicDB,/home/jlp/MATH.d/BIB.d/haddletonDB,/home/jlp/MATH.d/BIB.d/databaseDB,/home/jlp/MATH.d/BIB.d/evsciDB,/home/jlp/MATH.d/BIB.d/aiDB,/home/jlp/MATH.d/BIB.d/closureDB,/home/jlp/MATH.d/BIB.d/dataminingDB,/home/jlp/MATH.d/BIB.d/conceptDB,/home/jlp/MATH.d/BIB.d/accessDB,/home/jlp/MATH.d/BIB.d/logicDB,/home/jlp/MATH.d/BIB.d/closureappDB,/home/jlp/MATH.d/BIB.d/gtransDB,/home/jlp/MATH.d/BIB.d/topologyDB,/home/jlp/MATH.d/BIB.d/socialDB}
}

\end{document}